\documentclass{amsart}
\usepackage{amsmath}
\usepackage{amscd}
\usepackage{amssymb}
 \newtheorem{thm}{Theorem}[section]
 \newtheorem{cor}[thm]{Corollary}
 
 \newtheorem{prop}[thm]{Proposition}
 \theoremstyle{definition}

 \newtheorem*{rem}{Remark}
 
\newenvironment{enum}{\parindent0pt%
\begin{list}{}{%
\setlength{\itemindent}{0ex}
\setlength{\labelwidth}{15pt}
\setlength{\labelsep}{6pt}
\setlength{\leftmargin}{21pt}
\setlength{\listparindent}{0pt}
\setlength{\itemsep}{0ex}
\setlength{\topsep}{0ex}
\setlength{\parsep}{0.2em} 
}
}{\end{list}}

\newcommand {\fbG}{\frak{b}G}
\newcommand {\btG}{\boldsymbol{\beta}G}
\newcommand {\btS}{\boldsymbol{\beta}S}
\newcommand {\btX}{\boldsymbol{\beta}X}
\newcommand {\dG}{\widehat{G}}
\newcommand {\ddG}{\widehat{\!\widehat{G}}}
\newcommand {\dGd}{\widehat{G}_{\operatorname{d}}}
\newcommand {\ddGd}{\widehat{\widehat{G}_{\operatorname{d}}}}
\newcommand {\lv}{\left|}
\newcommand {\rv}{\right|}
\newcommand {\sbs}{\subseteq}
\newcommand {\mto} {\mapsto}
\newcommand {\cx}{\times}
\newcommand {\co}{\circ}
\newcommand {\cd}{\cdot}
\newcommand {\wtl}{\widetilde}
\newcommand {\imp}{\,\Rightarrow\,}
\newcommand {\Iff}{\ \Leftrightarrow\ }
\newcommand {\Eq}{\operatorname{Eq}}
\newcommand {\eps}{\varepsilon}
\newcommand {\vth}{\vartheta}
\newcommand {\Th}{\varTheta}
\newcommand {\vXi}{\varXi}
\newcommand {\vFi}{\varPhi}

\newcommand {\btN}{\boldsymbol{\beta}{\mathbb N}}
\newcommand {\bZ}{\mathbb Z}
\newcommand {\btZ}{\boldsymbol{\beta}{\mathbb Z}}
\newcommand {\bR}{\mathbb R}
\newcommand {\bC}{\mathbb C}
\newcommand {\bT}{\mathbb T}
\renewcommand{\:}{\colon}

\begin{document}

\title[Bohr and Stone-\v{C}ech  compactifications of abelian groups]
{The Bohr compactification of an abelian group as a quotient
of its Stone-\v{C}ech  compactification}

\author[P.~Zlato\v{s}]%
{Pavol Zlato\v{s}}

\address{%
{Pavol Zlato\v{s}}
\newline\indent
{\sl Faculty of Mathematics, Physics and Informatics}
\newline\indent
{\sl Comenius University}
\newline\indent
{\sl Mlynsk\'a dolina}
\newline\indent
{\sl 842\,48~Bratislava}
\newline\indent
{\sl Slovakia}
\newline\indent
{\tt zlatos@fmph.uniba.sk}}



\thanks{Author's research supported by the grant no.~1/0333/17
of the Slovak grant agency VEGA}

\begin{abstract}
We will prove that, for any abelian group $G$, the canonical (surjective and
continuous) mapping $\btG \to \fbG$ from the Stone-\v{C}ech compactification
$\btG$ of $G$ to its Bohr compactfication $\fbG$ is a homomorphism with respect
to the semigroup operation on $\btG$, extending the multiplication on $G$, and
the group operation on $\fbG$. Moreover, the Bohr compactification $\fbG$ is
canonically isomorphic (both in algebraic and topological sense) to the quotient
of $\btG$ with respect to the least closed congruence relation on $\btG$ merging
all the \textit{Schur ultrafilters} on $G$ into the unit of \,$G$.
\end{abstract}

\maketitle


\noindent
For any (discrete) semigroup $S$, its Stone-\v{C}ech compactification $\btS$
admits a semigroup operation extending the original multiplication on $S$ and
turning it into the universal compact right topological semigroup densely
extending $S$. The semigroups $\btS$ have proved their usefulness, versatility
and importance in various branches of mathematics, mainly in combinatorial
number theory and topological dynamics. In particular, the algebraic and
topological structure of the semigroups $\btN$ and $\btZ$ has been spectacularly
applied in proving a handful of striking combinatorial results in number theory.
The reader is referred to the monograph Hindman, Strauss \cite{HS} for a more
complete account.

Similarly, the Bohr compactification $\fbG$ of a locally compact abelian group $G$
is the universal compact abelian group densely extending $G$. It is of crucial
importance in harmonic analysis, mainly as the tool enabling to treat the almost
periodic functions on $G$ through their (continuous) extensions to $\fbG$. For
more details see, e.g., Hewitt, Ross \cite{HR1}, \cite{HR2}.

In the present paper we will bring to focus the relation between the Stone-\v{C}ech
and the Bohr compactification of any (discrete) abelian group $G$. Since the
Stone-\v{C}ech compactification $\btG$ is ``more universal'' than the Bohr
compactification $\fbG$, the embedding $G \to \fbG$ induces a canonical surjective
and continuous mapping $\xi\:\btG \to \fbG$. We will show that $\xi$ is a
homomorphism with respect to the semigroup operation on $\btG$, extending the
multiplication on $G$, and the group operation on $\fbG$. Then the set of all
pairs $(u,v) \in \btG \cx \btG$ such that $\xi(u) = \xi(v)$ is a closed congruence
relation on $\btG$. We will explicitly describe this relation as the least closed
congruence relation $\vXi(G)$ on $\btG$ merging all the \textit{Schur ultrafilters}
on $G$ (a~notion to be defined later on) into the unit of $G$. Thus the Bohr
compactification $\fbG$ is canonically isomorphic (both in algebraic and
topological sense) to the quotient $\btG/\vXi(G)$. As an intermediate result we
will show that, for any compact right topological semigroup $S$, the quotient
$S/\Th(S)$ of $S$ with respect to the least closed congruence relation $\Th(S)$
on $S$ containing all the pairs $(eu,u)$ where $e,u \in S$ and $e$ is an idempotent,
is a right topological \textit{group} with certain universal property.

In \cite{HS}, Chapter~21, various compactifications of semigroups $S$ were studied
and characterized as some particular quotients of the compact semigroups $\btS$.
However, the corresponding equivalence relations were always described in terms of
certain families of continuous functions from  $\btS$ to $\bR$ or $\bC$, obtained
as extensions of various kinds of ``almost periodic'' functions defined on~$S$.
In our approach, the corresponding congruence relations $\Th(\btG)$ and $\vXi(G)$
are described in terms of the inner algebraic and topological structure of the
Stone-\v{C}ech compactification $\btS$.

Analogous questions make sense within a more general framework: at least for
the class of locally compact abelian groups or even for more general classes
of topological groups. The former case is intended as a subject of our further
study \cite{Zlat}.

\section{Right topological semigroups and groups}\label{1}

\noindent
The reader is assumed to be familiar with the basic notions and results of general
topology. All the unexplained notions can be found, e.g., in Engelking~\cite{Eng}.
In what follows we will tacitly assume that all the topological spaces dealt
with are hausdorff. As a consequence, passing to a quotient $X/E$ of such a
topological space $X$, we will have to guarantee that the corresponding equivalence
relation, regarded as a subset $E \sbs X \cx X$, is closed.

In a semigroup $(S,\cd)$, the \textit{left shifts} $L_a\:S \to S$ and the
\textit{right shifts} $R_a\:S \to S$ are defined by $L_a(x) = ax$, $R_a(x) = xa$,
respectively, for $a,x \in S$. A semigroup $(S,\cd)$ endowed with a topology
$\tau$ is called a \textit{right topological semigroup} provided all the
right shifts $R_a\:S \to S$ are continuous. Using left shifts, the concept of a
\textit{left topological semigroup} can be defined analogously. A right
topological semigroup which is (algebraically) a group is called a \textit{right
topological group}, and similarly for left topological groups. If $(S,\cd,\tau)$
is both a left and right topological semigroup then we say that it is a
\textit{semitopological semigroup}. A \textit{topological semigroup} $(S,\cd,\tau)$
is a semigroup such that the multiplication $\cd : S \cx S \to S$ is (jointly)
continuous. A \textit{topological group} $(G,\cd,\tau)$ is a group which is a
topological semigroup such that the inverse map $J\:G \to G$, where $J(x) = x^{-1}$,
is continuous, as well (cf. Hindman, Strauss \cite{HS}).

Depending on context, we will omit the multiplication sign $\cdot$ or the symbol
of the topology $\tau$ or both in the notation of a (left or right) topological
semigroup $(S,\cd,\tau)$.

In what follows we will heavily depend on the following two results due to Ellis
(see, e.g., \cite{HS}, Theorem~2.5 and Corollary~2.39, respectively).

\begin{prop}\label{EllisIdempot}
Let $(S,\cd,\tau)$ be a compact right topological semigroup. Then $S$ contains
at least one idempotent element.
\end{prop}

\begin{prop}\label{EllisTopGrp}
Let $(G,\cd,\tau)$ be both a semitopological group and a locally compact
topological space. Then $(G,\cd,\tau)$ is a topological group.
\end{prop}

The notions of a congruence relation and of a quotient algebra make sense for
general algebras with any system of finitary operations (see, e.g., Burris,
Sankappanavar \cite{BS} or Gr\"{a}tzer \cite{Grt}), however, for our purpose,
it is sufficient to recall them just in the case of semigroups. An equivalence
relation $E$ on a semigroup $S$ is called a \textit{congruence relation} on $S$
if it is preserved by the semigroup operation, i.e., if for any
$(x_1,y_1), (x_2,y_2) \in E$ we have $(x_1 x_2,y_1 y_2) \in E$, as well. Denoting
by $[x]_E$ the coset of the element $x \in S$ with respect to $E$, the quotient
$S/E$ can be turned into a semigroup defining the operation on $S/E$ by
$$
[x]_E \cd [y]_E = [xy]_E,
$$
for $x,y \in S$. If $S$ is a group then any congruence relation $E$ on $S$ is
uniquely determined by the coset $[1]_E$ (which is a normal subgroup of $S$)
and $E$ is preserved by the inverse map, as well, i.e.,
$\bigl(x^{-1},y^{-1}\bigr) \in E$ for any $(x,y) \in E$, thus $S/E$ becomes
a group with the inverse map given by $[x]_E^{-1} = \bigl[x^{-1}]_E$. However,
for a semigroup this is not the case in general. If $\tau$ is a topology on $S$
and $E$ is \textit{closed} as subset of the product $S \cx S$ then $S/E$ endowed
with the quotient topology of $\tau/E$, i.e., the finest topology making the
canonical projection $x \mto [x]_E\:S \to S/E$ continuous, becomes a hausdorff
topological space. If $(S,\cd,\tau)$ is a right topological semigroup and
$E$ is a closed congruence relation on it, then the quotient $(S/E, \cd, \tau/E)$
is a right topological semigroup, as well. The properties ``being a left
topological semigroup'', ``being a topological group'', etc., are preserved
under the quotients with respect to closed congruence relations in an analogous way.

If $(S,\cd,\tau)$ is a right topological semigroup, then it is clear that the
full relation $S \cx S$ is a closed congruence relation on $S$ and the intersection
of any family $(E_i)_{i \in I}$ of closed congruence relations on $S$ is a closed
congruence relation on $S$. As a consequence, for every subset $D \sbs S \cx S$,
there exists the least closed congruence relation $\vFi$ on $S$ such that
$D \sbs \vFi$, namely
$$
\vFi = \bigcap\{E \mid
\text{$E$ is closed congruence relation on $S$ and $D \sbs E$}\}.
$$
In the discrete case, the description in purely algebraic terms of the least
congruence containing a given set $D \sbs S \cx S$ can be found in \cite{BS} or
\cite{Grt}.

Let $(S,\cd)$ be a semigroup and $\tau$ be a topology on $S$. We denote by
$\Th(S)$ the least closed congruence relation on $S$ containing all the
pairs $(eu,u)$ where $u \in S$ is an arbitrary element and $e \in S$ is an
idempotent. If $(S,\cd)$ has a unit element $1$ then, obviously, $\Th(S)$
coincides with the least closed congruence on $S$ containing all the pairs
$(e,1)$ where $e$ runs over all the idempotents in $S$.

Now, we can record the following easy consequence of the first of Ellis'
theorems (Proposition~\ref{EllisIdempot}).

\begin{thm}\label{IdempQuotient}
Let $(S,\cd,\tau)$ be a compact right topological semigroup. Then the
quotient $S/\Th(S)$ endowed with the quotient topology is a compact right
topological {\rm group}. Moreover, if \,$E$ is any closed congruence
relation on $S$, then $S/E$ is a right topological group if and only if
\,$\Th(S) \sbs E$.
\end{thm}

\begin{proof}
Let us denote $\Th = \Th(S)$; then $S/\Th$ with the quotient topology is a
compact (hausdorff) right compact topological semigroup. It suffices to show
that it is a group. Let $[u] = [u]_\Th \in S/\Th$ denote the coset of the
element $u \in S$ with respect to $\Th$. Obviously, for any idempotent $e \in S$,
the coset $[e]$ is the left unit in $S/\Th$. It remains to show that any coset
$[v]$ has a left inverse in $S/\Th$. Since the right side multiplication by $v$
is continuous, the set
$$
Sv = R_v[S] = \{sv \mid s \in S\}
$$
is a compact subsemigroup of $S$, hence it is a compact right topological semigroup,
as well. Thus, by the first Ellis' theorem (Proposition~\ref{EllisIdempot}), $Sv$
contains an idempotent of the form $e = uv$ for some $u \in S$. Then $[u]\,[v] = [e]$
is the unit in $S/\Th$ and $[u]$ is the left inverse of $[v]$. Thus $S/\Th$ is indeed
a group.

If $E$ is a closed congruence relation on $S$ then $S/E$ is a right topological
semigroup. If it is a group, then it is clear that all the idempotents in $S$ must
be sent to the unit element of $S/E$ by the canonical projection $S \to S/E$. Hence
$$
[eu]_E = [e]_E \cd [u]_E = [u]_E,
$$
and $(eu,u) \in E$, for any $u,e \in S$ whenever $e$ is an idempotent. Thus
$\Th \sbs E$. If $\Th \sbs E$ then $S/E$, as a homomorphic image
of the group $S/\Th$, is itself a group.
\end{proof}

The following universal property of the the canonical projection
$\vth\:S \to S/\Th(S)$ (and of the quotient $S/\Th(S)$) follows immediately
from the last theorem.

\begin{cor}\label{IdempQuotient-cor}
Let $(S,\cd,\tau)$ be a compact right topological semigroup,
$(G,\cd,\tau')$ be a right topological group and $\phi\:S \to G$ be a
continuous homomomorphism. Then there is a unique continuous homomorphism
$\phi'\:S/\Th(S) \to G$ such that
$$
\phi = \phi' \co \vth.
$$
\end{cor}

The following proposition is essentially just a more detailed reformulation
of the second of Ellis' theorems, quoted as Proposition~\ref{EllisTopGrp} here.

\begin{prop}\label{EllisTopGrp1}
Let $(G,\cdot,\tau)$ be a locally compact right topological semigroup. Then
the following conditions are equivalent:
\begin{enum}
\item[\rm (i)]
$(G,\cdot,\tau)$ is a topological group;

\item[\rm (ii)]
the inverse map $J\:G \to G$ is continuous;

\item[\rm (iii)]
$(G,\cdot,\tau)$ is a left topological group.
\end{enum}
\end{prop}

\begin{proof}
The implication (i)$\imp$(ii) is trivial, (iii)$\imp$(i) is the second of Ellis'
theorems (Proposition~\ref{EllisTopGrp}). Thus it suffices to prove
(ii)$\imp$(iii).

For any $a,x \in S$ we have $ax = (x^{-1}a^{-1})^{-1}$, i.e.,
$$
L_a = J \co R_{a^{-1}} \co J
$$
which shows the continuity of $L_a$.
\end{proof}

\section{The Stone-\v{C}ech compactification, Schur ultrafilters, \\
and the Bohr compactification}

\noindent
The Stone-\v{C}ech compactification of a set $X$, regarded as a discrete
topological space, consists of the set $\btX$ of all ultrafilters on $X$
endowed with the topology with the base formed by clopen sets of the form
$\{u \in \btX\mid A \in u\}$ where $A \sbs X$. Identifying each element $x \in X$
with the principal ultrafilter $\{A \sbs X\mid x \in A\}$, $X$~becomes embedded into
$\btX$ as a dense subset. Moreover, $\btX$ has the following universal property:
for any compact topological space $K$ and any (automatically continuous) mapping
$f\:X \to K$ there is a unique continuous mapping $\tilde f\:\btX \to K$ such that
$f(x) = \tilde f(x)$ for any $x \in X$, given by the $u$-limit
$$
\tilde f(u) = u\text{-}\!\lim_{x \in X} f(x) = \lim_{x \to u} f(x),
$$
for $u \in \btX$ (which is well defined by the compactness of $K$). Then $\tilde f$
is surjective if and only if $f[X]$ is dense in $K$ (see Hindman, Strauss \cite{HS}).

If $(S,\cd)$ is a (discrete) semigroup, then the semigroup operation can be
extended from $S$ to $\btS$ putting
$$
A \in uv \Iff \bigl\{s \in S\mid s^{-1}A \in v\bigr\} \in u,
$$
for $u,v \in \btS$, $A \sbs S$, where
$$
s^{-1}A = L_s^{-1}[A] = \{x \in S\mid sx \in A\}.
$$
Then $(\btS,\cd)$ is a compact right topological semigroup; however, the left
shifts $L_u\:\btS \to \btS$ are continuous just for the principal ultrafilters $u$,
i.e., elements of $S$. In general, $\btS$ is not commutative even if $S$ is
(see \cite{HS}, again).

Following Protasov \cite{Prot}, we call an ultrafilter $u \in \btS$ a
\textit{Schur ultrafilter} if for each $A \in u$ there exist $a,b \in A$ such
that $ab \in A$. It can be easily verified that every idempotent ultrafilter is
a Schur one. On the other hand, not every Schur ultrafilter is idempotent.
Namely it is known that that there are idempotent ultrafilters $e_1,e_2$ in the
Stone-\v{C}ech compactification $\btZ$ of the abelian group $(\bZ,+)$ such that
$e_1 + e_2$ is not idempotent. However, as proved by Protasov \cite{Prot},
Lemma~5.1, if $(G,+)$ is an abelian group, then the sum $u + v$ of any two Schur
ulrafilters in $\btG$ is Schur again. Hence, $e_1 + e_2$ is a Schur ultrafilter
which is not idempotent.

If $(G,\cd)$ is a group then, for each set $A \sbs G$, we denote
$A^{-1} = \bigl\{a^{-1} \mid a \in A\bigr\}$, as well as
$u^{-1} = \bigl\{A^{-1} \mid A \in u\bigr\}$ for any ultrafilter $u \in \btG$.
The mapping $u \mto u^{-1}$ is obviously a homeomorphism $\btG \to \btG$. We will
need the following result proved in \cite{Prot},
Lemma~5.2.

\begin{prop}\label{InvSchurUfi}
Let $G$ be a group. Then, for every ultrafilter $u \in \btG$, the ultrafilter
$uu^{-1}$ is Schur.
\end{prop}

Let $(G,\cd)$ be any group. We denote by $\vXi(G)$ the least closed congruence
relation on $\btG$ containing all the pairs $(u,1)$ where $u \in \btG$ is a Schur
ultrafilter. For the least closed congruence $\Th(\btG)$ on $\btG$ merging
together all the idempotents we obviously have $\Th(\btG) \sbs \vXi(G)$.

\begin{thm}\label{SchurQuotient}
Let $G$ be a group. Then the quotient $\btG/\vXi(G)$ is a compact topological
group.
\end{thm}

\begin{proof}
Let us abbreviate $\Th = \Th(\btG)$, \,$\vXi = \vXi(G)$, and denote by
$[u] = [u]_\vXi$ the coset of any ultrafilter $u \in \btG$ with respect to $\vXi$.
Since $\Th \sbs \vXi$ and $\btG/\Th$ is (algebraically) a group, so is its
homomorphic image $\btG/\vXi$. Obviously, $\btG/\vXi$ endowed with the quotient
topology, is a compact (hausdorff) space. Thus it suffices to show that $\btG/\vXi$
is indeed a \textit{topological} group. Clearly, $\btG/\vXi$ is a right topological
group. From Proposition~\ref{InvSchurUfi} it follows that the coset $[u^{-1}]$ is
the inverse element of the coset $[u]$. At the same time, the inverse map
$u \mto u^{-1}$ is continuous on $\btG$, hence the inverse map
$[u] \mto [u]^{-1} = [u^{-1}]$ is continuous on $\btG/\vXi$, as well. From
Proposition~\ref{EllisTopGrp1} it follows that $\btG/\vXi$ is a topological group.
\end{proof}

\begin{rem}
The fact that the quotient $\btG/\vXi(G)$ is (algebraically) a group could be proved
also directly, realizing that every right shift $(\btG)u$ contains an idempotent
hence a Schur ultrafilter, similarly as in the proof of Theorem~\ref{IdempQuotient},
and without using the fact that $\btG/\vXi(G)$ is a homomorphic image of
$\btG/\Th(\btG)$. However, Protasov's Lemma (Proposition~\ref{InvSchurUfi})
guarantees that this Schur ultrafilter has the particular form $u^{-1}u$, as well as
the continuity of the inverse map needed in order to allow for the application of
Proposition~\ref{EllisTopGrp1}.
\end{rem}

\begin{cor}\label{Abel}
Let $G$ be an abelian group. Then the quotient $\btG/\vXi(G)$ is a compact
{\rm abelian} group.
\end{cor}

\begin{proof}
It suffices to realize that the topological group $\btG/\vXi(G)$ contains a
dense abelian subgroup $G/\vXi(G)$.
\end{proof}

Let $G$ be any locally compact abelian (LCA) group. Its \textit{dual group}
$\dG$ consists of all continuous homomorphisms (characters) from $G$ to the
multiplicative group $\bT$ of all complex units, and, endowed with the
pointwise multiplication of characters and the compact-open topology, it is
an LCA group, too. In particular, $G$ is discrete if and only if $\dG$ is
compact and vice versa. The celebrated Pontryagin-van Kampen duality theorem
states that the canonical mapping $G \to \ddG$, sending any element $x \in G$
to the character $\hat x\:\dG \to \bT$ given by $\hat x(\gamma) = \gamma(x)$,
is an isomorphism of topological groups. Using this map, $G$ is identified with
its second dual $\ddG$, and the hat over $x$ is usually omitted (see, e.g.,
Hewitt, Ross~\cite{HR1}).

Using Pontryagin-van Kampen duality, the Bohr compactification $\fbG$ of any
LCA group $G$ can be defined as the dual group $\ddGd$ of $\dGd$, i.e., of the
dual group $\dG$ endowed with the discrete topology (see \cite{HR1}, \cite{HR2}).
Thus $\fbG$ consists of all homomorphisms $h\:\dG \to \bT$, and not just of the
continuous ones. Since $\dGd$ is discrete, its dual $\fbG$ is compact and the
canonical map $x \mto \hat x$ maps $G$ onto the dense subgroup $\ddG \cong G$
of $\fbG$. Alternatively, the Bohr compactification $\fbG$ can be characterized
through the following universal property: For any continuous homomorphism
$\phi\:G \to K$ from an LCA group $G$ to a compact topological group $K$ there
exists a unique continuous homomorphism $\phi^\sharp\:\fbG \to K$ such that
$\phi^\sharp\bigl(\hat x\bigr) = \phi(x)$ for all $x \in G$. Additionally,
$\phi^\sharp$ is surjective if and only if $\phi[G]$ is dense in $K$.

For any (discrete) abelian group $G$, the existence of the canonical map
$\btG \to \fbG$ follows from the universal property of $\btG$: There is a unique
continuous map $\xi\:\btG \to \fbG$ such that $\xi(x) = \hat x = x$ for each
$x \in G$. A more detailed description of this mapping uses the same universal
property of $\btG$ once again. Every character $\gamma \in \dG$, being a continuous
map $\gamma\:G \to \bT$, extends to a continuous map $\wtl\gamma\:\btG \to \bT$.
This mapping sends each ultrafilter $u \in \btG$ to the $u$-limit
$$
\wtl\gamma(u) = u\text{-}\!\lim_{x \in G} \gamma(x) = \lim_{x \to u} \gamma(x),
$$
which is well defined due to the compactness of the unit circle $\bT$.
At the same time, since the multiplication on $\bT$ is continuous, we have
$$
\wtl{\gamma\chi}(u) = \lim_{x \to u} (\gamma\chi)(x)
= \lim_{x \to u} \gamma(x)\,\lim_{x \to u} \chi(x) = \wtl\gamma(u)\,\wtl\chi(u),
$$
for $\gamma,\chi \in \dG$. That way every ultrafilter $u \in \btG$ induces a character
$\xi_u\:\dGd \to \bT$ given by
$$
\xi_u(\gamma) = \wtl\gamma(u),
$$
for $\gamma \in \dG$. Obviously, for a principal ultrafilter $x \in G$, we have
$\xi_x(\gamma) = \gamma(x) = \hat x(\gamma)$. Hence the assignment $u \mto \xi_u$
necessarily coincides with the canonical continuous surjective map
$\xi\:\btG \to \fbG$ induced by the inclusion map $G \to \fbG$.

Because of its nice logarithmic properties, we find more convenient to use the arc metric
$\lv\arg(x/y)\rv$ on the unit circle $\bT$ in the proofs of the next two propositions,
instead of the euclidian one. Obviously, they both induce the same topology on $\bT$.

\begin{prop}\label{HomCanMap}
Let $G$ be an abelian group. Then the canonical map $\xi\:\btG \to \fbG$ is
a homomorphism $(\btG,\cd) \to (\fbG,\cd)$.
\end{prop}

\begin{proof}
We will show that the ultralimit
$$
\xi_{uv}(\gamma) = \wtl\gamma(uv) = \lim_{x \to uv} \gamma(x)
$$
equals the product $\xi_u(\gamma)\,\xi_v(\gamma) = \wtl\gamma(u)\,\wtl\gamma(v)$,
for any $u,v \in \btG$ and $\gamma \in \dG$. Let $0 < \eps < \pi/2$. There are sets
$A \in u$, $B \in v$ such that $\lv\arg(\wtl\gamma(u)/\gamma(a)\rv < \eps$ for each
$a \in A$, as well as $\lv\arg(\wtl\gamma(v)/\gamma(b)\rv < \eps$ for each $b \in B$.
Then the set $C = AB$ obviously belongs to $uv$, and for any $c \in C$ we can find
$a \in A$, $b \in B$ such that $c = ab$. Then $\gamma(c) = \gamma(a)\,\gamma(b)$ and
\begin{align*}
\lv\arg\frac{\gamma(c)}{\xi_u(\gamma)\,\xi_v(\gamma)}\rv
&= \lv\arg\left(\frac{\gamma(a)}{\wtl\gamma(u)} \cd
    \frac{\gamma(b)}{\wtl\gamma(v)}\right)\rv \\
&\le \lv\arg\frac{\gamma(a)}{\wtl\gamma(u)}\rv +
      \lv\arg\frac{\gamma(b)}{\wtl\gamma(v)}\rv < 2\eps,
\end{align*}
showing that $\xi_{uv} = \xi_u\,\xi_v$.
\end{proof}

\begin{rem}
A quick inspection of the proof shows that we have proved a bit more. Namely,
$$
\xi_u(\gamma)\,\xi_v(\gamma) = \wtl\gamma(u)\,\wtl\gamma(v)
   = (u \odot v)\text{-}\!\lim_{x \in G} \gamma(x)\,,
$$
where $u \odot v$ denotes the filter on $\btG$ generated by all the sets of
the form $C = AB$, for $A \in u$, $B \in v$. As it is clear that
$u \odot v = v \odot u \sbs uv \cap vu$, we have
$$
\xi_{uv}(\gamma) = (u \odot v)\text{-}\!\lim_{x \in G} \gamma(x)
                 = \xi_{vu}(\gamma)\,.
$$
This has the neat consequence that, in spite of that the ultrafilters $uv$ and
$vu$ may differ, they still determine the same character of \,$\dGd$ (which, of
course, follows directly from the commutativity of the group $\fbG$ and homomorphy
of the canonical map $\btG \to \fbG$, as well).
\end{rem}

We will describe the closed congruence relation
$$
\Eq(\xi) = \{(u,v) \in \btG \cx \btG \mid \xi_u = \xi_v\}
= \bigcap_{\gamma\in\dG} \Eq(\wtl\gamma)
$$
on $\btG$, induced by the continuous surjective homomorphism $\xi$.

\begin{prop}\label{EqXi}
Let $G$ be an abelian group. Then
$$
\Eq(\xi) = \vXi(G).
$$
\end{prop}

\begin{proof}
First we show that $\xi_u(\gamma) = \wtl\gamma(u)= 1$ for every Schur ultrafilter
$u \in \btG$ and any $\gamma \in \dG$; then $\xi_u$ is the unit element in $\fbG$
and the inclusion $\vXi(G) \sbs \Eq(\xi)$ follows immediately. So let $u \in \btG$ be
Schur and $\gamma \in \dG$. Let $0 < \eps < \pi/3$. Then there is an $A \in u$ such
that
$$
\lv\arg\frac{\wtl\gamma(u)}{\gamma(c)}\rv <\eps
$$
for each $c \in A$. Let $a,b \in A$ be such that $ab \in A$, as well. Then
\begin{align*}
\lv\arg\xi_u(\gamma)\rv
&= \lv\arg\left(\frac{\wtl\gamma(u)}{\gamma(a)}
   \cd \frac{\wtl\gamma(u)}{\gamma(b)}
   \cd \frac{\gamma(ab)}{\wtl\gamma(u)}\right)\rv \\
&\le \lv\arg\frac{\wtl\gamma(u)}{\gamma(a)}\rv
   + \lv\arg\frac{\wtl\gamma(u)}{\gamma(b)}\rv
   + \lv\arg\frac{\gamma(ab)}{\wtl\gamma(u)}\rv < 3\eps.
\end{align*}
Hence $\lv\arg\xi_u(\gamma)\rv = 0$, i.e., $\xi_u(\gamma) = \wtl\gamma(u) = 1$.

Since $\fbG$ is a universal compactification of the group $G$, the reversed 
inclusion $\Eq(\xi) \sbs \vXi(G)$ follows from the fact that the quotient
$\btG/\vXi(G)$ is a compact topological group, established in Theorem
\ref{SchurQuotient}.
\end{proof}

As a consequence of Propositions~\ref{HomCanMap} and \ref{EqXi}, the mapping
$\btG/\vXi(G) \to \fbG$, induced by the canonical mapping $\xi\:\btG \to \fbG$,
is a bijective continuous homomorphism of topological groups. Since $\btG$ is
compact, it is a homeomorphism, too (see, e.g., Engelking~\cite{Eng}, Theorem~3.1.13).
That way we finally obtain

\begin{thm}\label{StoneCechQuotBohr}
Let $G$ be an abelian group and $\vXi(G)$ be the least closed congruence relation
on $\btG$ merging all the Schur ultrafilters $u \in \btG$ into the unit of \,$G$.
Then the mapping $\btG/\vXi(G) \to \fbG$, induced by  the canonical mapping
$\xi\:\btG \to \fbG$ given by $\xi_u(\gamma) = \wtl\gamma(u)$ for $u \in \btG$,
$\gamma \in \dG$, is an isomorphism of topological groups.
\end{thm}


\begin{thebibliography}{LLL}

\bibitem{BS}
S.~Burris, H.\,P.~Sankappanavar, \textit{A Course in Universal Algebra},
Springer, New York-Heidelberg-Berlin, 1981.

\bibitem{Eng}
R.~Engelking, \textit{General Topology},
PWN -- Polish Scientific Publishers, Warszawa, 1977.

\bibitem{Grt}
G.~Gr\"atzer, \textit{Universal Algebra} (2nd ed.),
Springer, New York-Heidelberg-Berlin, 1979.

\bibitem{HR1}
E.~Hewitt, K.\,A.~Ross, \textit{Abstract Harmonic Analysis {\rm I}},
Springer, Berlin-G\"otingen-Heidelberg, 1963.

\bibitem{HR2}
E.~Hewitt, K.\,A.~Ross, \textit{Abstract Harmonic Analysis {\rm II}},
Springer, Berlin-Heidelberg-New York, 1970.

\bibitem{HS}
N.~Hindman, D.~Strauss, \textit{Algebra in the Stone-\v{C}ech Compactification}
(2nd ed.), de~Gruyter, Berlin-Boston, 2012.

\bibitem{Prot}
I.\,V.~Protasov, \textit{Ul\!'trafil\!'try i topologii na gruppakh (Ultrafilters and
topologies on groups)}, Sibirsk. Mat. Zh. \textbf{34} (1993), 163--180 (in Russian).

\bibitem{Zlat}
P.~Zlato\v{s}, \textit{The Bohr compactification of a locally compact abelian group
as a quotient of its uniform Stone-\v{C}ech compactification}, in preparation.

\end{thebibliography}
\end{document}